\documentclass[12pt]{amsart}

\usepackage{graphicx}

\usepackage{subcaption} 

\usepackage{tikz-cd}
\usepackage{hyperref}
\hypersetup{bookmarks=true,
	unicode=true,
	colorlinks=true,
	citecolor=black,
	linkcolor=black,
	urlcolor=black,
	plainpages=false,
	pdfpagelabels=true}

 \usepackage{url}	
 \allowdisplaybreaks 

\usepackage{tikz-cd}
\usepackage{pgf}

\usepackage{xcolor}

\usepackage{comment} 

\usepackage[all]{xy}

\newtheorem{teo}{Theorem}[section]
\newtheorem{theorem}[teo]{Theorem}

\newtheorem{lemma}[teo]{Lemma}

\theoremstyle{definition}
\newtheorem{definition}[teo]{Definition}

\newtheorem{example}[teo]{Example}

\newtheorem{thmnn}{Theorem}

\theoremstyle{remark}
\newtheorem{remark}{Remark}

\numberwithin{figure}{section}

\begin{document}

\title[Lefschetz theorems for acyclic categories]{Lefschetz fixed-object and fixed-morphism theorems for acyclic categories}
\thanks{The authors were partially supported by  Xunta de Galicia ED431C 2019/10 with FEDER funds and by MINECO-Spain research project PID2020-MTM-114474GB-100.
}


\author[S. Castelo]{%
	Samuel Castelo Mourelle
}

 \address{%
           	Samuel Castelo Mourelle
            \\
              Departamento de Matem\'aticas, Universidade de Santiago de Compostela, 15782-SPAIN}
               \email{samuel.castelo@rai.usc.es}

\author[E. Mac\'ias]{%
	Enrique Macías-Virgós 
}

 \address{%
           	Enrique Macías-Virgós \\
              CiTMAga, Departamento de Matem\'aticas, Universidade de Santiago de Compostela, 15782-SPAIN}
               \email{quique.macias@usc.es}

\author[D. Mosquera]{%
	David Mosquera-Lois 
}

 \address{%
           	David Mosquera-Lois \\
              Departamento de Matemáticas, Universidade de Vigo, SPAIN }
               \email{david.mosquera.lois@usc.es \\
               david.mosquera.lois@uvigo.es}


\begin{abstract} We introduce two novel complementary notions of the Lefschetz number for a functor from a finite acyclic category to itself and we prove a Lefschetz fixed-object theorem and a Lefschetz fixed-morphism theorem. 
\end{abstract}

\keywords{acyclic category, Lefschetz number, Lefschetz fixed-point theorem}

\subjclass[2020]{
Primary: 55M20 
Secondary: 55U10
}

\maketitle


\section*{Introduction}

Let $X$ be a finite simplicial structure, possibly more general than a simplicial complex. The Lefschetz number of a simplicial map  $f\colon X\to X$ is defined as
\begin{equation}\label{eq:def_lefschetz}
    \mathcal{L}(f)=\sum_{i \ge 0} (-1)^i \mathrm{trace}(H_i(f)),
\end{equation} where $H_i(f): H_i(X) \to H_i(X)$ denotes the induced morphism in homology with rational coefficients.

The Lefschetz number stands as a crucial invariant within Algebraic Topology. Beyond its role as an extension of the Euler-Poincaré characteristic, $\chi(X)=L(\textup{id}_X)$, the Lefschetz number serves as a connecting link between Topology and Dynamics. It ensures the fulfillment of fixed-point theorems, establishing a significant bridge between these two mathematical domains. 

In 1979 Backlawski and Björner proved a Lefschetz fixed point theorem for posets, which play an important role as objects of study in Combinatorial Algebraic Topology:
\begin{thmnn}[Lefschetz fixed point theorem for posets 
 \protect{\cite[Theorem 1.1]{Bjorner}}]\label{lefeuler}
    Let $X$ be a finite poset and $f\colon X\rightarrow X$ an order-preserving map. Then
    \[\mathcal{L}(f)=\chi(X^f).\]
    Moreover, if $\mathcal{L}(f)\neq 0$, then there exists $x\in X$ such that $f(x)=x$, that is to say there exists at least one fixed point of $f$. 
\end{thmnn}

Nowadays, acyclic (or loop-free) categories are gaining much importance in Combinatorial Algebraic Topology (see for example \cite{Babson, Kozlov, May, Delucchi, Delucchi2, Tanaka}).

The purpose of this short note is to extend Theorem \ref{lefeuler} to finite acyclic categories. We define two complementary notions of the Lefschetz number for a functor $F\colon C \rightarrow C$ from a finite acyclic category to itself and we prove a Lefschetz fixed-object theorem (Theorem \ref{teoremagrande}) and a Lefschetz fixed-morphism theorem (Theorem \ref{fixedmorph}) for such a functor.



\section{Background}
We begin by introducing some notions regarding acyclic categories and fixing the notation. For a detailed exposition we refer the reader to \cite{Kozlov, May}.

 A \textit{(finite) acyclic category} or \textit{(finite) loop-free category} is a (finite) small category (the collection of arrows is a (finite) set) such that only the identity morphisms have inverses and the unique morphism from an object to itself is the identity. We will represent and denote the acyclic categories by omitting the identity arrows. The acyclic categories along with the functors between them form a category, which we call \textit{AC}. All acyclic categories we work with will be assumed to have a finite number of morphisms.


Posets are examples of acyclic categories. For a poset $(X,\leq)$, the objects are the elements of the poset and there is a single morphism between two objects $x$ and $y$ if and only if $x\leq y$. As a consequence, we can see the category of posets and order-preserving maps between them, which we denote $Poset$, as a subcategory of $AC$. Moreover, for an acyclic category $C$, there is a partial order on the set of objects of $C$ given by $x\leq y$ if and only if there is an arrow from $x$ to $y$. The resulting poset is denoted by $R(C)$.  Let $C$ and $D$ be acyclic categories and let $F\colon C\rightarrow D$ be a functor. The map,
    $R(F)\colon R(C)\rightarrow R(D)$ 
    defined by  $R(F)(x)=F(x)$ is an order-preserving map between posets. Therefore, $R\colon \textup{ACat} \to \text{Posets}$ is a functor.
    
Let $C$ be an acyclic category. The \textit{nerve} of $C$, denoted $\Delta (C)$, is a regular trisp (polyhedral complex whose cells are simplices, also referred as generalized simplicial complex or regular $\Delta-$complex depending on the literature) defined as follows (see \cite[Definition 10.4.]{Kozlov} and \cite[Definition 2.48]{Kozlov}):
  \begin{itemize}
        \item The set of vertices is the set of objects of $C$.
        \item For $k\geq 1$, there is one $k$-simplex $\sigma$ for each chain of $k$ non-identity composable morphisms $x_0\overset{\alpha_1}{\rightarrow}x_1\overset{\alpha_2}{\rightarrow}\dots\overset{\alpha_k}{\rightarrow} x_k$. 
    \end{itemize}

Intuitively, regular trisps generalize simplicial complexes by allowing several different simplices to share all its vertices. Let $C$ and $D$ be two acyclic categories and let $F\colon C\rightarrow D$ be a functor. There is a regular trisp map $\Delta(F)\colon \Delta(C)\rightarrow \Delta(D)$ given as follows. For a composable morphism chain $x_0\overset{\alpha_1}{\rightarrow}x_1\overset{\alpha_2}{\rightarrow}\dots\overset{\alpha_k}{\rightarrow} x_k$ in $\Delta(C)$, $\Delta(F)$ maps it to the composable morphism chain $F(x_0)\overset{F(\alpha_1)}{\rightarrow} F(x_1)\overset{F(\alpha_2)}{\rightarrow}\dots\overset{F(\alpha_k)}{\rightarrow} F(x_k)$ in $\Delta(D)$. We use $\textup{RT}$ to denote the category of regular trisps.

We recall the \textit{face poset} functor $\mathcal{F}\colon RT \to Posets$ for the particular case of regular trisps and regular trisps maps coming from acyclic categories by the functor $\Delta$ since they are the only ones we will consider.

For a regular trisp $\Delta(C)$, $\mathcal{F}(\Delta(C))$ is the poset of simplices of $\Delta(C)$ ordered by inclusion: $\beta \leq \tau$ if $\beta$ is a face of $\tau$. For $\Delta(F)\colon \Delta(C)\rightarrow \Delta(D)$, $\mathcal{F}(\Delta(F))\colon \mathcal{F}(\Delta(C))\rightarrow \mathcal{F}(\Delta(D))$ maps the $k$-simplex (seen as an element of $\mathcal{F}(\Delta(C))$) $x_0\overset{\alpha_1}{\rightarrow}x_1\overset{\alpha_2}{\rightarrow}\dots\overset{\alpha_k}{\rightarrow} x_k$ in $\Delta(C)$, to $F(x_0)\overset{F(\alpha_1)}{\rightarrow} F(x_1)\overset{F(\alpha_2)}{\rightarrow}\dots\overset{F(\alpha_k)}{\rightarrow} F(x_k)$ in $\mathcal{F}(\Delta(D))$.

We define the subdivision functor for acyclic categories $\textup{sd}\colon ACat \to Posets$ as the composition $\textup{sd}=\mathcal{F}\circ \Delta$.

 We summarize these functors we will use in this work in Diagram \ref{eq:contexto}.
\begin{equation}
	\label{eq:contexto}
	\begin{tikzcd}
		\text{ACat}  \arrow[rr,"{\mathrm{sd}}", swap] \arrow[dr,"{\Delta}",  swap] & & \text{Posets} \arrow[ll,"{i}",,hook,  shift right=1ex, swap]  \\
		 & \text{RT}  \arrow[ur,"{\mathcal{F}}", swap]        
	\end{tikzcd}
\end{equation}

The homology of the acyclic category $C$, which we denote by $H(C)$ is just the homology $H(\Delta(C))$ of its nerve. Moreover, the face poset functor (consequently the subdivision functor) induces isomorphisms at the level of homology.

We introduce some notation. For an acyclic category $C$, we denote by $O$ the set of objects and by $M$ the set of morphisms. For a functor  $F\colon C\rightarrow C$, we consider the pair of maps between objects and between arrows: $F_O\colon O(C)\rightarrow O(C)$ and $F_M\colon M(C)\rightarrow M(C)$.
\begin{itemize}
    \item The \textit{fixed-object set} of $F\colon C \rightarrow C$ is the subset $$O(C)^F=\lbrace x\in O(C)\colon F_O(x)=x\rbrace\subseteq O(C).$$
    \item The \textit{fixed-arrow set} of $F\colon C \rightarrow C$ is the subset, $$M(C)^F=\lbrace \alpha\in M(C)\colon F_M(\alpha)=\alpha\rbrace\subseteq M(C).$$
    \item The \textit{fixed subcategory} of $F\colon C \rightarrow C$ is the acyclic category $$C^F=(O(C)^F,M(C)^F).$$
\end{itemize}

        
    




\section{The Lefschetz fixed-object theorem}
We introduce two novel complementary notions of Lefschetz number for a functor $F\colon C \rightarrow C$ between finite acyclic categories and we prove a Lefschetz fixed-object theorem.
\begin{definition}
       Let $C$ be a finite acyclic category and let  $F\colon C \rightarrow C$ be a functor. We define its \textit{Lefschetz number} as
     $\mathcal{L}(F)=\mathcal{L}(\Delta(F))$
     and its \textit{R-Lefschetz number}  as
    $\mathcal{L}_R(F)=\mathcal{L}(R(F))$.
\end{definition}

\begin{remark}\label{igualesacy}
    Observe that for a functor between acyclic categories $F\colon C \rightarrow C$, $\mathcal{L}(F)=\mathcal{L}(\textup{sd}(F))$ and that for an order-preserving map between posets  $F\colon P \rightarrow P$, $\mathcal{L}(F)=\mathcal{L}_R(F)$.
\end{remark}

\begin{theorem}[Lefschetz fixed-object theorem] \label{teoremagrande}
    Let $C$ be a finite acyclic category and $F\colon C \rightarrow C$ be a functor. Then:
    \begin{enumerate}
        \item[(a)] $\mathcal{L}_R(F)=\chi(R(C)^{R(F)}).$
        \item[(b)] $\mathcal{L}(F)=\chi(C^F).$
    \end{enumerate}
    Moreover, if $\mathcal{L}_R(F)\neq 0$ or $\mathcal{L}(F)\neq 0$, then there exists at least one object $x\in O(C)$ such that $F(x)=x$, i.e. $O(C)^F\neq \emptyset$.
\end{theorem}

Theorem \ref{teoremagrande} depends on a crucial observation. Consider Figure \ref{barycencatar2}. We represent a functor $F\colon C\rightarrow C$ between acyclic categories and the subdivision functor  $\textup{sd}(F)\colon \textup{sd}(C)\to \textup{sd}(C)$. We see that if $\textup{sd}(F)$ fixes $(\gamma,\beta) \in \textup{sd}(C)$, then  $\textup{sd}(F)$ fixes all elements below $(\gamma,\beta)$ in $\textup{sd}(C)$. 
\begin{figure}[h]
        \centering
        \includegraphics[width=1\textwidth]{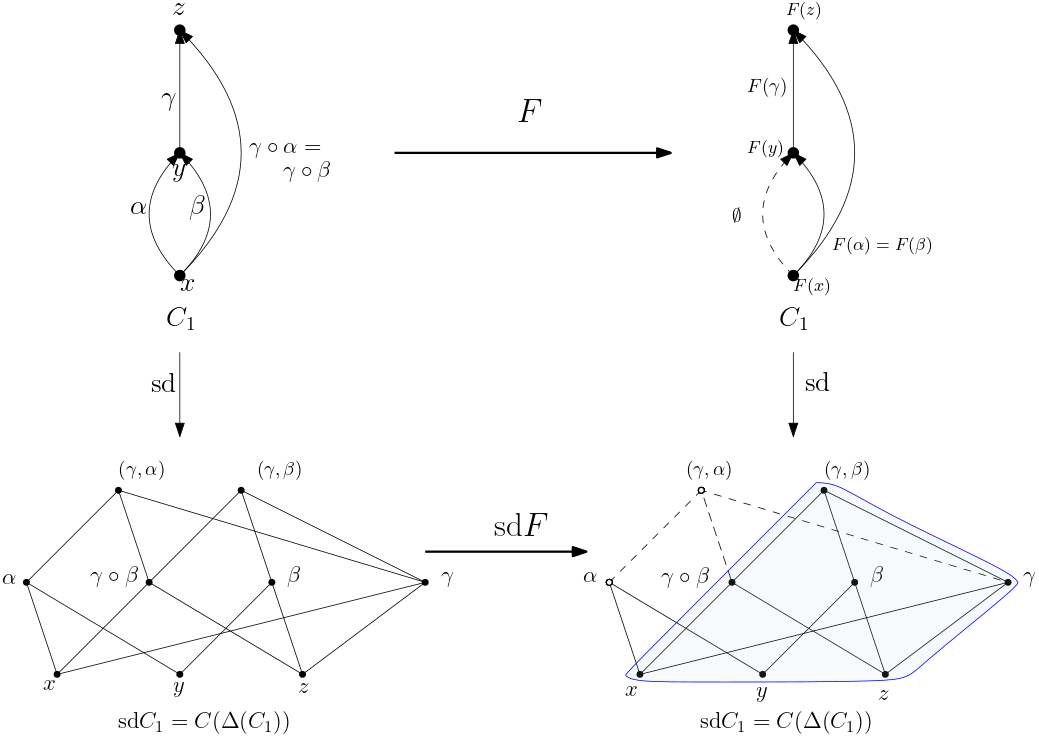}
        \caption{A functor $F\colon C\rightarrow C$ between acyclic categories and the subdivision functor  $\textup{sd}(F)\colon \textup{sd}(C)\to \textup{sd}(C)$.}
        \label{barycencatar2}
    \end{figure}

This is not particular to this concrete example but a general fact. Let us state this as a lemma.

\begin{lemma}\label{caraabixo}
   Let $C$ be a finite acyclic category and let $F\colon C\rightarrow C$ be a functor.   If  $\textup{sd}(F)\colon \textup{sd}(C)\to \textup{sd}(C)$ fixes an object $\sigma\in \textup{sd}(C)$ ($\textup{sd}(F)(\sigma)=\sigma$), then $\textup{sd}(F)(\tau)=\tau$ for all $\tau \leq \sigma$. 
\end{lemma}
\begin{proof}
    Consider $\sigma\in \textup{sd}(C)$, $\sigma=x_0\overset{\alpha_1}{\rightarrow}x_1\overset{\alpha_2}{\rightarrow}\dots\overset{\alpha_k}{\rightarrow} x_k$.  If  $\textup{sd}(F)\colon \textup{sd}(C)\to \textup{sd}(C)$ fixes $\sigma\in \textup{sd}(C)$, then it fixes $x_1\overset{\alpha_2}{\rightarrow}x_2\overset{\alpha_2}{\rightarrow} \dots \overset{\alpha_k}{\rightarrow} x_k$, $x_0\overset{\alpha_1}{\rightarrow}x_1 \dots \overset{\alpha_{k-1}}{\rightarrow} x_{k-1}$ and $x_0\overset{\alpha_1}{\rightarrow}x_1 \dots x_{j-1}\overset{\alpha_{j+1}\circ \alpha_j}{\longrightarrow} x_{j+1} \dots\overset{\alpha_k}{\rightarrow} x_k$ for any $1\leq j\leq k-1$. By applying this argument iteratively a finite number of times it holds that for all $\tau \leq \sigma$, $\textup{sd}(F)(\tau)=\tau$.
\end{proof}

\begin{proof}[Proof of Theorem \ref{teoremagrande}]
    First of all, observe that (a) is a direct consequence of Theorem \ref{lefeuler}.  Let us prove (b). It is enough to check the following chain of equalities:
     \[\mathcal{L}(F)\overset{(1)}{=}\mathcal{L}(\textup{sd}(F))\overset{(2)}{=}\chi(\textup{sd}(C)^{\textup{sd}(F)})\overset{(3)}{=}\chi(\textup{sd}(C^F))\overset{(4)}{=}\chi(C^F).\]
     Equality (1) follows from Remark \ref{igualesacy}. Equality (2) is consequence of Theorem \ref{lefeuler}. Equality (4) follows from the homological invariance of Euler-Poincaré characteristic. It remains to prove Equality (3). Suppose $\sigma\in\textup{sd}(C)^{\textup{sd}(F)} $, where  $\sigma=x_0\overset{\alpha_1}{\rightarrow}x_1\overset{\alpha_2}{\rightarrow}\dots\overset{\alpha_k}{\rightarrow} x_k$. By Lemma \ref{caraabixo}, $\sigma\in \textup{sd}(C^F)$. Conversely, assume that $\sigma\in \textup{sd}(C^F)$ where $\sigma=x_0\overset{\alpha_1}{\rightarrow}x_1\overset{\alpha_2}{\rightarrow}\dots\overset{\alpha_k}{\rightarrow} x_k$. By construction, \begin{align*}
         \textup{sd}(F)(x_0\overset{\alpha_1}{\rightarrow}x_1\overset{\alpha_2}{\rightarrow}\dots\overset{\alpha_k}{\rightarrow} x_k)&=F(x_0)\overset{F(\alpha_1)}{\rightarrow}F(x_1)\overset{F(\alpha_2)}{\rightarrow}\dots\overset{F(\alpha_k)}{\rightarrow} F(x_k)\\
         &=x_0\overset{\alpha_1}{\rightarrow}x_1\overset{\alpha_2}{\rightarrow}\dots\overset{\alpha_k}{\rightarrow} x_k.
     \end{align*} 
     Let us prove now the last part of the theorem. Observe that the elements of $R(C)$ coincide with the objects of $C$ while $F$ and $R(F)$ fix the same objects. Therefore, by Theorem \ref{lefeuler}, if $\mathcal{L}_R(F)\neq 0$, then there exists at least one object $x\in O(C)$ such that $F(x)=x$. 
     Assume now that $\mathcal{L}(F)\neq 0$. Then, by Theorem \ref{lefeuler} there exists  $\sigma\in\textup{sd}(C)^{\textup{sd}(F)}$, $\sigma=x_0\overset{\alpha_1}{\rightarrow}x_1\overset{\alpha_2}{\rightarrow}\dots\overset{\alpha_k}{\rightarrow} x_k$, such that $\textup{sd}(F)(\sigma)=\sigma$. By Lemma \ref{caraabixo}, $\textup{sd}(F)(x_0)=x_0$.  Finally, $F(x_0)=\textup{sd}(F)(x_0)$ by definition, which finishes the proof.
\end{proof}

We show that for a functor $F\colon C\rightarrow C$ between acyclic categories,  $\mathcal{L}(F)$ and $\mathcal{L}_R(F)$ provide complementary information regarding the fixed objects of $F\colon C\rightarrow C$.

\begin{example}
Consider the acyclic categories $C$ and $D$ depicted in Figure \ref{ex:funtores}:
\begin{figure}[h]
        \centering
        \includegraphics[width=0.7\textwidth]{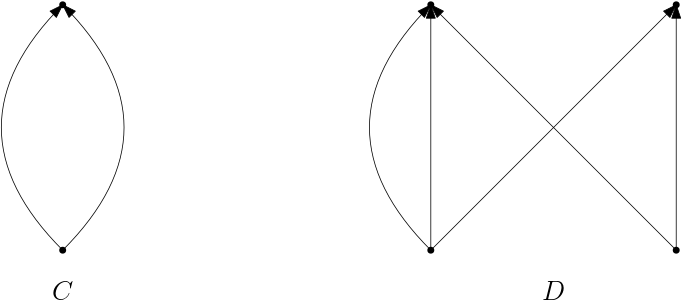}
        \caption{Categories $C$ and $D$.}
        \label{ex:funtores}
\end{figure}
We define functors $F\colon C\to C$ and $G\colon D\to D$ 
 which are the identity on each category. Therefore, they fix the respective categories (both arrows and objects).  We have $\mathcal{L}_R(F)=\chi(R(C))=2-1=1$, $\mathcal{L}(F)=\chi(C)=2-2=0$, $\mathcal{L}_R(G)=\chi(R(D))=4-4=0$ and $\mathcal{L}(G)=\chi(D)=4-5=-1$.
\end{example}

\section{Lefschetz fixed-morphism theorem}
We begin by fixing some notation in order to state and prove the Lefschetz fixed-morphism theorem. Let $P$ be a poset. A morphism is {\em indecomposable} if it cannot be represented as a composition of two nonidentity morphisms. The poset $P$ is {\em graded} if there is a function $\textup{deg}\colon P \to \mathbb{Z}$  such that whenever $\alpha< \beta$ is an indecomposable morphism, we have $\textup{deg}(\beta) = \textup{deg}(\alpha) + 1$.  Moreover, the {\em degree} of $\alpha$, denoted, $\textup{deg}(\alpha)=i$, coincides with the maximum of the lengths of the chains $\alpha_0<\alpha_1<\cdots<\alpha$  where $\alpha_0$ is a minimal element of $P$. 

Let $P$ be a graded poset.  We say that the element $\alpha$ belongs to its \textit{i-th layer}, denoted $\alpha\in P_i$, if  $\textup{deg}(\alpha)=i$. We define $P^{\hspace{1mm}\widehat{i}}$ as the deletion of the $i$-th layer from the poset, that is, the subposet whose elements are $P\setminus P_i.$
     Analogously, we define $P^{\hspace{1mm}\widehat{\leq i}}$  as the deletion of all the $j$-th layers with $j\leq i$ from the poset $P$.

     Let $F\colon P\rightarrow P$ be an order-preserving map. Then $F$ induces order-preserving maps 
     \[F^{\hspace{1mm}\widehat{i}}\colon P^{\hspace{1mm}\widehat{i}}\setminus F^{-1}(P_i)\rightarrow P^{\hspace{1mm}\widehat{i}}\setminus F^{-1}(P_i)\]
     and 
     \[F^{\hspace{1mm}\widehat{\leq i}}\colon P^{\hspace{1mm}\widehat{\leq i}}\setminus F^{-1}\left(\bigcup_{j\leq i}P_j\right)\rightarrow P^{\hspace{1mm}\widehat{\leq i}}\setminus F^{-1}\left(\bigcup_{j\leq i}P_j\right)\label{defmal}\] 
     given by $F^{\hspace{1mm}\widehat{i}}(\alpha)=F(\alpha)$ and $F^{\hspace{1mm}\widehat{\leq i}}(\alpha)=F(\alpha)$, respectively. 
Finally, let us denote 

\[{\textup{sd}}_F^{\hspace{1mm}\widehat{\leq i}}=\textup{sd}(C)^{\hspace{1mm}\widehat{\leq i}}\setminus (\textup{sd}({F)^{\hspace{1mm}\widehat{\leq i}}})^{-1} \left( \bigcup_{j\leq i}\textup{sd}(C)_j)\right)\] 
so there is an induced order-preserving map $$\textup{(sd} F)^{\hspace{1mm}\widehat{ \leq i}} \colon {\textup{sd}}_F^{\hspace{1mm}\widehat{\leq i}} \to {\textup{sd}}_F^{\hspace{1mm}\widehat{\leq i}}.$$

\begin{theorem}[Lefschetz fixed-morphism theorem]\label{fixedmorph}
    Let $C$ be a finite acyclic category and let $F\colon C\rightarrow C$ be a functor. If  $\mathcal{L}(\textup{(sd} F)^{\hspace{1mm}\widehat{ \leq i}})\neq 0$, then for all $1\leq k\leq i+1$ there exists at least one chain of composable morphims $x_0\overset{\alpha_1}{\rightarrow}x_1\overset{\alpha_2}{\rightarrow}\dots\overset{\alpha_k}{\rightarrow} x_k$ in $C$ such that $F(\alpha_j)=\alpha_j$ for $1\leq j\leq k$ and   $F(x_0\overset{\alpha_1}{\rightarrow}x_1\overset{\alpha_2}{\rightarrow}\dots\overset{\alpha_k}{\rightarrow} x_k)=x_0\overset{\alpha_1}{\rightarrow}x_1\overset{\alpha_2}{\rightarrow}\dots\overset{\alpha_k}{\rightarrow} x_k$. Particularly, if $\mathcal{L}(\textup{(sd} F)^{\hspace{1mm}\widehat{ 0}})\neq 0$, then there exists at least one non-identity morphism $\alpha$ in $C$ such that  $F(\alpha)=\alpha$,  i.e, $M(C)^F\neq \emptyset$. 
\end{theorem}

\begin{proof}
      Consider the order-preserving map $\textup{(sd} F)^{\hspace{1mm}\widehat{ \leq i}} \colon {\textup{sd}}_F^{\hspace{1mm}\widehat{\leq i}} \to {\textup{sd}}_F^{\hspace{1mm}\widehat{\leq i}}$. 
      By Theorem \ref{lefeuler},  $\textup{(sd} F)^{\hspace{1mm}\widehat{\leq i}}$ has fixed elements.

      The elements of ${\textup{sd}}_F^{\hspace{1mm}\widehat{\leq i}}$ are of the form $x_0\overset{\alpha_1}{\rightarrow}x_1\overset{\alpha_2}{\rightarrow}\dots\overset{\alpha_k}{\rightarrow} x_k$ with $k\geq i+1$ where each $x_i$ is an object of $C$ and each $\alpha_i\colon x_{i-1}\to x_i$ is a morphism or arrow in $C$.     Therefore, since $(\textup{sd} F)^{\hspace{1mm}\widehat{\leq i}}$ fixes an element $x_0\overset{\alpha_1}{\rightarrow}x_1\overset{\alpha_2}{\rightarrow}\dots\overset{\alpha_k}{\rightarrow} x_k$  in ${\textup{sd}}_F^{\hspace{1mm}\widehat{\leq i}}$, then $F$ fixes at least one chain of morphisms of length $k$ in $C$. By Lemma \ref{caraabixo}, $(\textup{sd} F)^{\hspace{1mm}\widehat{\leq i}}$ fixes  $x_1\overset{\alpha_2}{\rightarrow}x_2 \dots \overset{\alpha_k}{\rightarrow} x_k$, $x_0\overset{\alpha_1}{\rightarrow}x_1 \dots \overset{\alpha_{k-1}}{\rightarrow} x_{k-1}$ and $x_0\overset{\alpha_1}{\rightarrow}x_1 \dots x_{j-1}\overset{\alpha_{j+1}\circ \alpha_j}{\longrightarrow} x_{j+1} \dots\overset{\alpha_k}{\rightarrow} x_k$ in ${\textup{sd}}_F^{\hspace{1mm}\widehat{\leq i}}$ for any $1\leq j\leq k-1$. Applying Lemma \ref{caraabixo} repeatedly, it follows that for all $1\leq k\leq i+1$ there exists at least one chain of composable morphims $x_0\overset{\alpha_1}{\rightarrow}x_1\overset{\alpha_2}{\rightarrow}\dots\overset{\alpha_k}{\rightarrow} x_k$ in $C$ such that $F(\alpha_j)=\alpha_j$ for $1\leq j\leq k$ and   $F(x_0\overset{\alpha_1}{\rightarrow}x_1\overset{\alpha_2}{\rightarrow}\dots\overset{\alpha_k}{\rightarrow} x_k)=x_0\overset{\alpha_1}{\rightarrow}x_1\overset{\alpha_2}{\rightarrow}\dots\overset{\alpha_k}{\rightarrow} x_k$.
\end{proof}

\begin{example}
   Let us consider the category and functor presented in Figure \ref{fig:fixedmorph}. Computing $\mathcal{L}((\textup{sd} F)^{\widehat{0}})=4-3=1$ we determine that $F$ has fixed arrows (for example $\beta$ and $\gamma \circ \beta$).
    \begin{figure}[h!]
        \centering
        \includegraphics[width=1\textwidth]{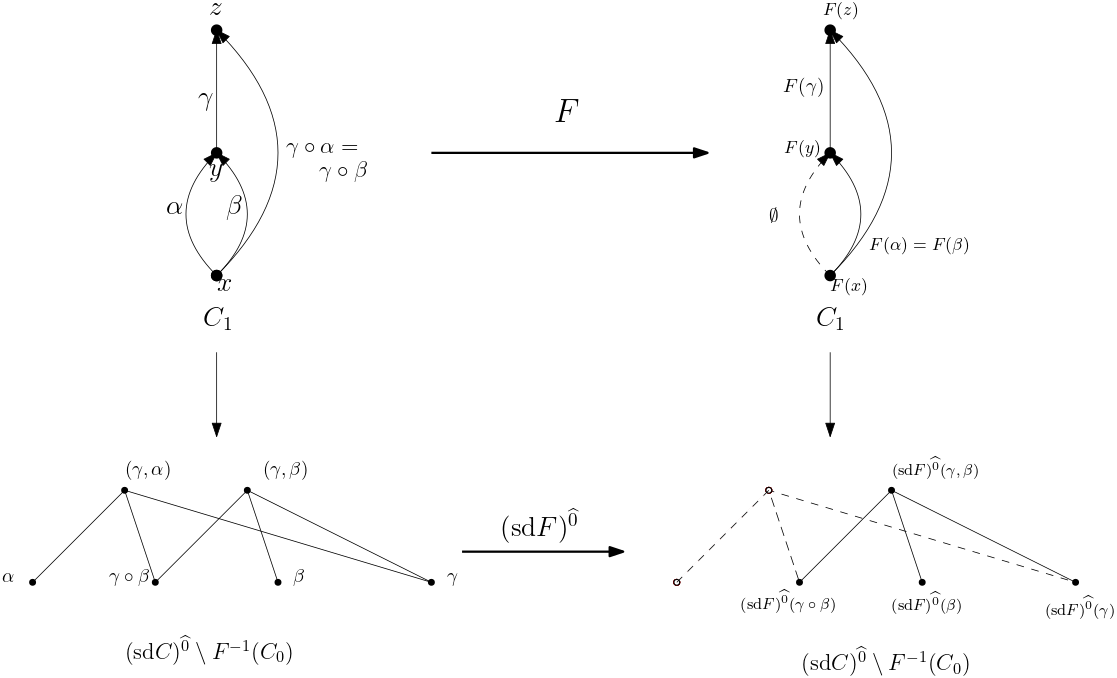}
        \caption{Example of Theorem \ref{fixedmorph}.}
        \label{fig:fixedmorph}
    \end{figure}
\end{example}

\newpage





\end{document}